\documentclass[12pt]{amsart}
\usepackage{amsmath, amsthm, amscd, amsfonts}

\newtheorem{theorem}{Theorem}[section]

\newtheorem{proposition}[theorem]{Proposition}
\newtheorem{corollary}[theorem]{Corollary}
\theoremstyle{definition}
\newtheorem{definition}[theorem]{Definition}

\theoremstyle{remark}

\newtheorem{remark}[theorem]{Remark}
\numberwithin{equation}{section}

\begin{document}
\title[Generalized inverses and polar decomposition]
{Generalized inverses and polar decomposition of unbounded
regular operators on Hilbert $C^*$-modules}
\author{M. Frank}
\address{Michael Frank, \newline Hochschule f\"{u}r Technik,
Wirtschaft und Kultur (HTWK) Leipzig,
Fachbereich IMN, Gustav-Freytag-Strasse 42A, D-04277 Leipzig,
Germany } \email{mfrank@imn.htwk-leipzig.de}
\author{K. Sharifi}
\address{Kamran Sharifi, \newline Department of Mathematics,
Shahrood University of Technology, P. O. Box 3619995161-316,
Shahrood, Iran} \email{sharifi.kamran@gmail.com and
sharifi@shahroodut.ac.ir}

\subjclass{Primary 46L08; Secondary 47L60, 46C05}
\keywords{Hilbert $C^*$-module, unbounded operator, polar
          decomposition, generalized inverses,
          $C^*$-algebras of compact operators}

\begin{abstract}
In this note we show that an unbounded regular operator $t$ on
Hilbert $C^*$-modules over an arbitrary $C^*$ algebra $
\mathcal{A}$ has polar decomposition if and only if the closures
of the ranges of $t$ and $|t|$ are orthogonally complemented, if
and only if the operators $t$ and $t^*$ have unbounded regular
generalized inverses. For a given $C^*$-algebra $ \mathcal{A}$ any
densely defined $\mathcal A$-linear closed operator $t$ between
Hilbert $C^*$-modules has polar decomposition, if and only if any
densely defined $\mathcal A$-linear closed operator $t$ between
Hilbert $C^*$-modules has generalized inverse, if and only if
$\mathcal A$ is a $C^*$-algebra of compact operators.
\end{abstract}
\maketitle

\section{Introduction.}
In the theory of $C^*$-algebras, an important role is played by
the spaces which are modules over a $C^*$-algebra and are
equipped with a structure which is like an inner product but which,
instead of being scalar-valued as in the case of Hilbert spaces,
takes its values in the $C^*$-algebra. Such modules
are called (pre-)Hilbert $C^*$-modules. Let us quickly recall the
definition of a Hilbert $C^*$-module.

A (left) {\it pre-Hilbert $C^*$-module} over a (not necessarily
unital) $C^*$-algebra $\mathcal{A}$ is a left $\mathcal{A}$-module
$E$ equipped with an $\mathcal{A}$-valued inner product $\langle
\cdot , \cdot \rangle : E \times E \to \mathcal{A}$,
which is $\mathcal A$-linear in the first variable and has the
properties:
$$ \langle x,y \rangle=\langle y,x \rangle ^{*},
   \ \ \ \langle x,x \rangle \geq 0 \ \ {\rm with} \
   {\rm equality} \ {\rm if} \ {\rm and} \ {\rm only} \
   {\rm if} \ x=0.
$$
We always suppose that the linear structures of $\mathcal A$
and $E$ are compatible.

A pre-Hilbert $\mathcal{A}$-module $E$ is called a {\it Hilbert $
\mathcal{A}$-module} if $E$ is a Banach space with respect to the
norm $\| x \|=\|\langle x,x\rangle \|_{\mathcal A}^{1/2}$. If $E$,
$F$ are two Hilbert $\mathcal{A}$-modules then the set of all
ordered pairs of elements $E \oplus F$ from $E$ and $F$ is a
Hilbert $\mathcal{A}$-module with respect to the $\mathcal
A$-valued inner product $\langle (x_{1},y_{1}),(x_{2},y_{2})
\rangle= \langle x_{1},x_{2}\rangle_{E}+\langle y_{1},y_{2}
\rangle _{F}$. It is called the {\it orthogonal sum of $E$ and
$F$}. A pre-Hilbert $\mathcal A$-module $E$ of a pre-Hilbert
$\mathcal A$-module $F$ is an orthogonal summand if $E \oplus
E^\bot = F$, where $E^\bot$ denotes the orthogonal complement of
$E$ in $F$. If $F$ is a Hilbert $ \mathcal{A}$-module and $F=E
\oplus E^\bot$ then $E$ and $E^\bot$ are necessarily  Hilbert $
\mathcal{A}$-submodules (cf. \cite{WEG}, Lemma 15.3.4). Some
interesting results about orthogonally complemented submodules can
be found in \cite{FR2}, \cite{FR3}, \cite{MAG}, \cite{SCH}. For
the basic theory of Hilbert $C^*$-modules we refer to the books
\cite{LAN}, \cite{M-T} and to some chapters of \cite{WEG}.

As a convention, throughout the present paper we assume
$\mathcal{A}$ to be an arbitrary $C^*$-algebra (i.e. not necessarily
unital). Since we deal with bounded and unbounded operators at the
same time we simply denote bounded operators by capital letters and
unbounded operators by lower case letters. We use the denotations
$Dom(.)$, $Ker(.)$ and $Ran(.)$ for domain, kernel and range of
operators, respectively.

Suppose $E,\ F$ are Hilbert $\mathcal{A}$-modules. We denote the
set of all bounded $\mathcal{A}$-linear maps $T: E \to F$ for which
there is a map $T^*: F \to E$  such that the equality $ \langle Tx,y
\rangle _{F} = \langle x,T^*y \rangle _{E}$ holds for any $ x \in
E,\ y \in F$ by $B(E,F)$. The operator $T^*$ is called the {\it
adjoint operator} of $T$.

The polar decomposition is a useful tool that represents an
operator as a product of a partial isometry and a positive
element. It is well known that every bounded operator on Hilbert
spaces has polar decomposition. In general bounded adjointable
operators on Hilbert $C^*$-modules do not have polar composition,
but Wegge-Olsen has given a necessary and sufficient
condition for bounded adjointable operators to admit polar
decomposition. He has proved that a bounded adjointable operator
$T$ has polar decomposition if and only if $\overline{Ran(T)}$ and
$\overline{Ran(|T|)}$ are orthogonal direct summands (cf.
\cite{WEG}, Theorem 15.3.7).

Let us review the polar decomposition of densely defined closed
operators on Hilbert spaces. Suppose $H$ and $H^{'}$ are Hilbert
spaces and $t : Dom(t) \subseteq H \to H^{'}$ is a densely
defined closed operator, then there exists a partial isometry $
\mathcal{V} \in B(H,H^{'})$ such that
$$
    t=\mathcal{V}|t|, \ \ Ker(t)=Ker( \mathcal{V})
$$
where $|t|:=(t^*t)^{1/2}$ (cf. \cite{KAT}, VI. Section 2.7).
Furthermore, every densely defined closed operator $t : Dom(t)
\subseteq H \to H^{'}$ has a densely defined, closed generalized
inverse, i.e. there exists a densely defined closed operator $s$
such that $tst=t$, $sts=s$, $(ts)^*= \overline{ts}$ and $(st)^*=
\overline{st}$ (cf. \cite{PYT}, Lemma 12).

In \cite{GUL} Gulja\v{s} lifts the above facts to the densely
defined closed operators on Hilbert $C^*$-modules over arbitrary
$C^*$-algebras $K(H)$ of all compact operators on a Hilbert space
$H$ of arbitrary cardinality. In fact he has found a bijective
operation-preserving map between the space of all densely defined
closed operators on Hilbert $K(H)$-modules and the space of all
densely defined closed operators on a suitable Hilbert space,
so he may lift certain properties of operators from Hilbert spaces
to Hilbert $K(H)$-modules (cf. \cite{GUL}, Theorems 2.4, 3.1, 3.3).

In the present note we give a necessary and sufficient condition
for unbounded regular operators to admit polar decomposition. In
fact we will prove that an unbounded regular operator $t$ on
Hilbert $C^*$-modules over an arbitrary $C^*$ algebra
$\mathcal{A}$ has polar decomposition if and only if
$\overline{Ran(t)}$ and $\overline{Ran(|t|)}$ are orthogonally
complemented, if and only if the operators $t$ and $t^*$ have
unbounded regular generalized inverses.

Some interesting characterizations of an arbitrary $C^*$-algebra
of compact operators (i.e. of a $C^*$-algebra that admits a
faithful $*$-representation in the set of all compact operators on
a certain Hilbert space) have been given in \cite{ARV},
\cite{FR1}, \cite{F-S}, \cite{MAG}, \cite{SCH}. Beside the work
of these authors we give other descriptions of the $C^*$-algebra
of compact operators via the above properties.

\newcounter{cou001}

\section{Preliminaries}

In this section we recall some definitions and basic facts about
regular operators on Hilbert $\mathcal{A}$-modules. These
operators were first introduced by Baaj and Julg in \cite{B-J}.
More details and properties can be found in chapters 9 and 10 of
\cite{LAN}, and in the papers \cite{F-S},  \cite{KK}, \cite{PAL},
\cite{KUS}, \cite{WOR}.

Let $E, F$  be Hilbert $\mathcal{A}$-modules, we will use the
notation $t : Dom(t) \subseteq E \to F$ to indicate that $t$ is
an $\mathcal{A}$-linear operator whose domain $Dom(t)$ is a dense
submodule of $E$ (not necessarily identical with $E$) and whose
range is in $F$. Given $t,s:Dom(t),Dom(s) \subseteq E \to F$, we
write $s \subseteq t$ if $Dom(s)\subseteq Dom(t)$ and $s(x)=t(x)$
for all $x \in Dom(s)$. A densely defined operator $t: Dom(t)
\subseteq E \to F$ is called {\it closed} if its graph
$G(t)=\{(x,t(x)): \ x \in Dom(t)\}$ is a closed submodule of the
Hilbert $\mathcal{A}$-module $E \oplus F$. If $t$ is closable,
the operator $s : Dom(s) \subseteq E \to F$ with the property
$G(s)= \overline{G(t)}$ is called the {\it closure} of $t$ denoted
by $s= \overline{t}$.
The operator $ \overline{t}$ is the smallest closed operator that
contains $t$.

A densely defined operator $t: Dom(t) \subseteq E \to F$
is called {\it adjointable} if it possesses a densely defined map
$t^*: Dom(t^*) \subseteq F \to E$ with the domain
$$
   Dom(t^*)= \{y \in F : {\rm there} \ {\rm exists} \ z \in E \
   {\rm such} \ {\rm that} \ \langle t(x),y \rangle _{F} = \langle
   x , z \rangle _{E} \ {\rm for} \ {\rm any} \ x \in Dom(t) \}
$$
which satisfies the property $\langle t(x),y \rangle_{F} = \langle
x,t^*(y) \rangle_{E}$,  for any $x \in Dom(t), \ y \in Dom(t^*)$.
This property implies that $t^*$ is a closed $\mathcal{A}$-linear
map.

\begin{remark}
Recall that the composition of two densely defined operators
$t,s$ is the unbounded operator $ts$ with $Dom(ts)=\{x\in Dom(s)
: \ s(x)\in Dom(t) \}$ given by $(ts)(x)=t(s(x))$ for all  $x\in
Dom(ts)$. The operator $ts$ is not necessarily densely defined.
Suppose two densely defined operators $t,s$ are adjointable, then
$s^*t^* \subseteq (ts)^*$. If $T$ is a bounded adjointable
operator, then $s^*T^* = (Ts)^*$.
\end{remark}

A densely defined closed $\mathcal{A}$-linear map $t: Dom(t)
\subseteq E \to F$ is called {\it regular} if it is adjointable
and the operator $1+t^*t$ has a dense range. We denote the set of
all regular operators from $E$ to $F$ by $R(E,F)$. A criterion of
regularity via the graph of densely defined operators has been
given in \cite{F-S}. In fact a densely defined operator $t$ with 
a densely defined adjoint operator is
regular if and only if its graph is orthogonally complemented in
$E \oplus F$ (cf. \cite{F-S}, Corollary 2.2). If $t$ is regular
then $t^*$ is regular and $t=t ^{**}$, moreover $t^*t$ is regular
and selfadjoint (cf.~\cite{LAN}, Corollaries 9.4, 9.6 and
Proposition 9.9). Define $Q_{t}=(1+t^*t)^{-1/2}$ and
$F_{t}=tQ_{t}$, then $Ran(Q_{t})=Dom(t)$,  $0 \leq Q_{t} \leq 1$
in $B(E,E)$ and $F_{t}\in B(E,F)$ (cf.~\cite{LAN}, chapter 9). The
bounded operator $F_{t}$ is called the bounded transform (or
$z$-transform) of the regular operator $t$. The map $t\to F_{t}$
defines a bijection $$ R(E,F) \to \{ T \in B(E,F):\| T \|\leq 1 \
\, {\rm and} \ \ Ran(1- T^* T ) \ {\rm is} \ {\rm dense} \ in \ F
\}, $$ (cf. \cite{LAN}, Theorem 10.4). This map is
adjoint-preserving, i.e. $F_{t}^*=F_{t^*}$, and for the bounded
transform $F_{t}=tQ_{t}=t(1+t^*t)^{-1/2}$ we have $\|F_{t}\|\leq
1$ and $$ t=F_{t}(1-F_{t}^*F_{t})^{-1/2} \ {\rm and} \
Q_{t}=(1-F_{t}^*F_{t})^{1/2} \, . $$

For a regular operator $t \in R(E):=R(E,E)$ some usual
properties may be defined. A regular operator $t$ is called
{\it normal} iff $Dom(t)=Dom(t^*)$ and $\langle t(x),t(x)
\rangle=\langle t^*(x),t^*(x) \rangle$ for any $x \in Dom(t)$.
The operator $t$ is called {\it selfadjoint} iff $t^*=t$, and
$t$ is called {\it positive} iff $t$ is normal and
$\langle t(x),x \rangle \geq 0$ for any $x \in Dom(t)$.
Remarkably, a regular operator $t$ is selfadjoint (resp.,
positive) iff its bounded transform $F_t$ is selfadjoint
(resp., positive), cf. \cite{KUS,LAN}. Moreover, both
$t$ and $F_t$ have the same range and the same kernel.
A regular operator $t$ has closed range if and only if
its adjoint operator $t^*$ has closed range, and then for
$|t|:=(t^*t)^{1/2}$ the orthogonal sum decompositions $E=
Ker(t) \oplus Ran(t^*)= Ker(|t|) \oplus \overline{Ran(|t|)}$,
$F= Ker(t^*) \oplus Ran(t) = Ker(|t^*|) \oplus
\overline{Ran(|t^*|)}$ exist, cf. Proposition 1.2 of
\cite{F-S} and Result 7.19 of \cite{KUS}.

\begin{remark}
Let $t$ be a regular operator on an arbitrary Hilbert
$\mathcal{A}$-module $E$, and $F _{t} $ and $Q_{t}$ be as above
then one can see that $F_{t} \cdot p(F^*_{t}F_{t})=p(F_{t}F^*_{t})
\cdot F_{t}$ \ for any polynomial $p$ and, hence, by continuity
for any $p$ in $\mathbf{C}([0,1])$. In particular,
$F_{t}(1-F^*_{t}F_{t}) ^{1/2}=(1-F_{t}F^*_{t}) ^{1/2}F_{t}$ and so
by the equalities $Q_{t}=(1-F^*_{t}F_{t})^{1/2}$ and
$Q_{t^*}=(1-F_{t}F^*_{t}) ^{1/2}$ we have
$tQ_{t}^2=Q_{t^*}tQ_{t}$.
\end{remark}

Before closing this section we would like to define the concept of
generalized (or pseudo-) inverses of unbounded regular operators,
which is motivated by the definitions of densely defined closed
operators in \cite{GUL} and \cite{PYT}.

\begin{definition}
Let $t \in R(E,F)$ be a regular operator between two Hilbert
${\mathcal{A}}$-modules $E,F$ over some fixed $C^*$-algebra
${\mathcal{A}}$. A regular operator $s \in R(F,E)$ is called
the generalized inverse of $t$ if $tst=t$, $sts=s$, $(ts)^*=
\overline{ts}$ and $(st)^*= \overline{st}$.
\end{definition}

If a regular operator $t$ has a generalized inverse $s$, then
the above definition implies that $Ran(t) \subseteq Dom(s)$
and $Ran(s) \subseteq Dom(t)$. Note, that bounded
$\mathcal{A}$-linear operators may admit generalized inverses
in the set of regular operators even if they do not admit any
bounded generalized inverse operator. For examples, consider
contractive operators on Hilbert spaces with dense, but
non-closed range. Moreover, for bounded linear operators on
Hilbert spaces, for example, the property to admit polar
decomposition does not imply the property to admit a bounded
generalized inverse. More surprising are the results for
unbounded operators described in the next section.

\section{The polar decomposition and generalized inverses}

\begin{theorem} \label{thm_polar_decomp}
If $E,F$ are arbitrary Hilbert $\mathcal{A}$-modules over a
$C^*$-algebra of coefficients $\mathcal{A}$ and $t \in R(E,F)$
denotes a regular operator then the following conditions are
equivalent:

\begin{list}{(\roman{cou001})}{\usecounter{cou001}}
\item $t$ has a unique polar decomposition $t= \mathcal{V}|t|$,
where $\mathcal{V} \in B(E,F)$ is a partial isometry for which
$Ker(\mathcal{V})=Ker(t)$, $Ker(\mathcal{V^*})=Ker(t^*)$,
$Ran(\mathcal{V})=\overline{Ran(t)}$,
$Ran(\mathcal{V^*})=\overline{Ran(|t|)}$. That is
$\overline{Ran(t)}$ and $\overline{Ran(|t|)}=\overline{Ran(t^*)}$
are final and initial submodules of the partial isometry
$\mathcal{V}$, respectively.

\item $E=Ker(|t|) \oplus \overline{Ran(|t|)}$ and $F=Ker(t^*) \oplus
\overline{Ran(t)}$.

\item $t$ and $t^*$ have unique generalized inverses which are
adjoint to each other, $s$ and $s^*$.
\end{list}

\noindent
In this situation, $\mathcal{V^*}\mathcal{V}=\overline{t^*s^*}$ is
the projection onto $\overline{Ran(|t|)} = \overline{Ran(t^*)}$,
$\mathcal{V}\mathcal{V^*}=\overline{ts}$ is the projection onto
$\overline{Ran(t)}$, and $\mathcal{V^*}\mathcal{V}|t|=|t|$,
$\mathcal{V^*}t=|t|$ and $\mathcal{V}\mathcal{V^*}t=t$.
\end{theorem}

\begin{proof}
(i) $\Rightarrow$ (ii) Let $\mathcal{V} \in B(E,F)$ be a partial
isometry satisfying condition (i). Then the identity map
of $E$ can be written as the sum of two orthogonal projections
$I-\mathcal{V^*}\mathcal{V}$ and $\mathcal{V^*}\mathcal{V}$. By
Result 7.19 of \cite{KUS} we have
$$
   Ran(I-\mathcal{V^*}\mathcal{V})=Ker(
   \mathcal{V})=Ker(t)=Ker(|t|) \, ,
$$
$$
   Ran( \mathcal{V^*}\mathcal{V})=Ran(
   \mathcal{V^*})=\overline{Ran(|t|)} \, .
$$
So we get $E=Ran(I-\mathcal{V^*}\mathcal{V})\oplus
Ran(\mathcal{V^*}\mathcal{V})=Ker(|t|)\oplus\overline{Ran(|t|)}$.
Similarly, $F$ can be decomposed as $F=Ker(t^*) \oplus
\overline{Ran(t)}$.

(ii) $\Rightarrow$ (i) Let $t \in R(E,F)$ be a regular operator
and $F_{t}$ be its bounded transform. Then Proposition 1.2 of
\cite{F-S}, Result 7.19 of \cite{KUS} and Proposition 3.7 of
\cite{LAN} imply that
$$
  Ker(t)=Ker(F_{t})=Ker(|F_{t}|)=Ker(|t|),
  Ker(t^*)=Ker(F_{t^*})
$$
$$
  \overline{Ran(|F_{t}|)}=\overline{Ran(F_{t}^*)}= \overline{Ran(t^*)}
  =\overline{Ran(|t|)}, \overline{Ran(t)}=\overline{Ran(F_{t})} \, .
$$
From the above equalities and (ii) we have $E=Ker(|F_{t}|)
\oplus \overline{Ran(|F_{t}|)}$, $F=Ker(F_{t}^*) \oplus
\overline{Ran(F_{t})}$. Now Theorem 15.3.7 of \cite{WEG} implies
that there exists a unique partial isometry $ \mathcal{V} \in
B(E,F)$ such that $F_{t}=\mathcal{V}|F_{t}|$ and
$$
Ker(
\mathcal{V})=Ker(F_{t}), \ \ Ker( \mathcal{V^*})=Ker(F_{t}^*),
$$
$$
Ran(\mathcal{V})=\overline{Ran(F_{t})},
Ran(\mathcal{V^*})=\overline{Ran(|F_{t}|)}.
$$

Therefore
$$
  Ker( \mathcal{V})=Ker(t),  Ker(\mathcal{V^*})=
  Ker(t^*) \, ,
$$
$$
  Ran( \mathcal{V})=\overline{Ran(t)},
  Ran(\mathcal{V^*})=\overline{Ran(|t|)} \, .
$$
Furthermore, by Remark 2.2 we have $F_{t}=
\mathcal{V}|F_{t}|=\mathcal{V}(t^*Q_{t^*}tQ_{t})^{
1/2}=\mathcal{V}(t^*tQ_{t}^{2}) ^{1/2}$ that is
$tQ_{t}=\mathcal{V}(t^*t)^{1/2} Q_{t}.$ But $Q_{t}:E
\longrightarrow Ran(Q_{t})=Dom(t)$ is invertible, so
$t=\mathcal{V}(t^*t)^{1/2}= \mathcal{V}|t|$.

(ii) $\Rightarrow$ (iii) Recall, that $Ker(|t|) = Ker(t)$ and
$Ran(|t|)=Ran(t^*)$. We set $Dom(s):=Ran(t) \oplus Ker(t^*)$
and define $s: Dom(s)\subseteq F \longrightarrow E$ by
$s(t(x_{1}+x_{2})+x_{3})=x_{1}$, for all $x_{1} \in Dom(t) \cap
\overline{Ran(t^*)}$, $x_{2} \in Dom(t) \cap Ker(t)$ and $x_{3}
\in Ker(t^*)$. This definition is correct since $E = Ker(t)
\oplus \overline{Ran(t^*)}$ by supposition. Then $s$ is an
$\mathcal{A}$-linear module map the domain of which
is a dense $\mathcal{A}$-submodule of $F$, since
$F=\overline{Ran(t)}\oplus Ker(t^*)$.

For each $x \in Dom(t)$ with $x=x_{1}+x_{2}$, $x_{1} \in
Dom(t) \cap  \overline{Ran(t^*)}$, $x_{2} \in Dom(t) \cap Ker(t)$
we have $tst(x)=ts(t(x_1+x_2)+0)=t(x_1)=t(x_1+x_2)$, i.e.
$tst=t$. Similarly, for each $x=t(x_{1}+x_{2})+x_{3} \in Dom(s)$
such that $x_{1} \in Dom(t) \cap  \overline{Ran(t^*)}$, $x_{2} \in
Dom(t) \cap Ker(t)$ and $x_{3} \in Ker(t^*)$, we have $sts(x)=
st(x_{1}+x_{2})=x_{1}=s(x)$, and so $sts=s$. Now we are going
to derive the properties of Definition 2.3 to demonstrate that
$s$ is a regular operator and the generalized inverse of
the operator $t$.

By the definition of $s$, the equality $ts(t(x_1+x_2)+x_3) =
t(x_1)=t(x_1+x_2)$ holds. Consequently, the operator $ts$ acts on
$Ran(t)$ as the identity operator, and on the orthogonal
complement $Ran(t)^\bot$ as the zero operator. By continuity, the
closure $\overline{ts}$ of $ts$ is the projection onto the
orthogonal summand $\overline{Ran(t)}$ of $F$. So,
$(ts)^*=\overline{(ts)^*}=(\overline{ts})^*=\overline{ts}$.
Analogously, the operator equality $(st)^*=\overline{st}$ can be
derived, and $(st)^*$ can be shown to be the projection onto the
orthogonal summand $\overline{Ran(s)}$ of $E$.

We set $Dom( \tilde{s} ):=Ran(t^*) \oplus Ker(t)$ and define the
module map $\tilde{s}:Dom( \tilde{s} )\subseteq E \longrightarrow
F $ by $ \tilde{s}(t^*(y_{1}+y_{2})+y_{3})=y_{1}$, for any $y_{1}
\in Dom(t^*) \cap \overline{Ran(t)}$, $y_{2} \in Dom(t^*) \cap
Ker(t^*)$ and $y_{3} \in Ker(t)$. Then $\tilde{s}$ is an
$\mathcal{A}$-linear module map which domain $Dom(\tilde{s})$ is
a dense $\mathcal{A}$-submodule of $E$ since
$E=\overline{Ran(|t|)}\oplus Ker(|t|)=\overline{Ran(t^*)}\oplus
Ker(t)$. We also have $t^* \tilde{s} t^*=t^*$ and $\tilde{s}t^*
\tilde{s}=\tilde{s}$. Similarly, $\overline{ t^* \tilde{s}}= (t^*
\tilde{s})^* $ and $ \overline{\tilde{s}t^* }=( \tilde{s}t^*)^*$
are orthogonal projections onto $\overline{Ran(t^*)}$ and $
\overline{Ran(\tilde{s})}$, respectively.

We prove that $s$ is a regular operator and $s^*=\widetilde{s}$.
Consider the isometry $U \in B(E \oplus F, F \oplus E)$ by
$U(x,y)=(y,x)$, then by Proposition 9.3 of \cite{LAN} we have $ F
\oplus E=U G(t) \oplus G(-t^*)$ and so
\begin{eqnarray*}
F \oplus E &=& \{ (t(x_1),x_1): x_1 \in Dom(t)\cap
                \overline{Ran(t^*)} \} \oplus
                \{(0,y_3): y_3 \in Ker(t) \} \\
           & & \oplus \,\, \{ (y_1,-t^*(y_1)): y_1 \in Dom(t^*)\cap
                \overline{Ran(t)} \} \oplus
                \{(x_3,0): x_3 \in Ker(t^*) \} \\
           &=& \{( t(x_1)+x_3 \, ,x_1) \, : x_1 \in Dom(t)
                \cap \overline{Ran(t^*)}~,~ x_3 \in Ker(t^*) \}\\
           & & \oplus \,\, \{( y_1,-t^*(y_1)-y_3) : y_1 \in Dom(t^*)
                \cap \overline{Ran(t)}~,~ y_3 \in Ker(t) \}\\
           &=&  G(s) \oplus VG( \widetilde{s}) \, ,
\end{eqnarray*}

where $ V \in B(E \oplus F, F \oplus E)$ is an isometry defined
by $V(x,y)=(y,-x)$. The equality $ F \oplus E= G(s) \oplus V
G(\tilde{s})$ and Corollary 2.2 of \cite{F-S} imply that the
operator $s$ is adjointable, closed and the range of $1+s^*s$ is
dense in $F$. In particular $s^*= \tilde{s}$. Clearly, the regular
operators $s$ and $s^*$ with the properties of generalized
inverses of the operators $t$ and $t^*$, respectively, are unique.

(iii)$\Rightarrow$(ii) Let $s$ be generalized inverse of $t$, so
that $tst=t$, $sts=s$, $(ts)^*= \overline{ts}$ and $(st)^*=
\overline{st}$. Therefore $(ts) ^{2}=(ts)$ and $Ran(ts)=Ran(t)$,
what implies that $ \overline{ts}$ is an orthogonal projection on $
\overline{Ran(t)}$, i.e. $\overline{Ran(t)}$ is orthogonally
complemented. By the hypothesis $s^*$ is the generalized inverse
of $t^*$, therefore $\overline{t^*s^*}$ is an orthogonal
projection onto $\overline{Ran(t^*)}=\overline{Ran(|t|)}$, i.e.
$\overline{Ran(|t|)}$ is orthogonally complemented.

Note that $ \mathcal{V^*}\mathcal{V}$ is the orthogonal
projection onto $\overline{Ran(|t|)}$ so $
|t|=\mathcal{V^*}\mathcal{V}|t|$. This together with the polar
decomposition of $t$, gives $ \mathcal{V^*}t=|t|$ and $
t=\mathcal{V}\mathcal{V^*}t$.
\end{proof}

The previous theorem and its proof imply some interesting
results as follows:

\begin{corollary} \label{cor_transfer}
If $t \in R(E,F)$ and $F_{t}$ is its bounded transform, then t
has polar decomposition $t= \mathcal{V}|t|$ if and only if
$F_{t}$ has polar decomposition $F_{t}= \mathcal{V}|F_{t}|$, if
and only if  $F_{t}$ has polar decomposition $F_{t}= {\mathcal V}
F_{ |t| }$, for the partial isometry $ \mathcal{V}$ which was
introduced in Theorem \ref{thm_polar_decomp}.
\end{corollary}

For the proof, just recall that $t$ and $F_t$ have the same kernel
and the same range and that $F_{t^*}=F_t^*$. Note that $Q_{ |t| }
= Q_{t}$ and so $F_{ |t| } = | F_{t} |$.

\begin{corollary}
An operator $t \in R(E,F)$ has polar decomposition
$t= \mathcal{V}|t|$ if and only if its adjoint $t^*$ has
polar decomposition $t^*= \mathcal{V^*}|t^*|$,  for the partial
isometry $\mathcal{V}$ which was introduced in Theorem
\ref{thm_polar_decomp}.
\end{corollary}

\begin{proof}If $t= \mathcal{V}|t|$ then $F_{t}=
\mathcal{V}|F_{t}|$, so $F_{t}^*=|F_{t}|\mathcal{V^*}$. Define
$G:=\mathcal{V} |F_{t}| \mathcal{V^*}$, then $G$ is selfadjoint
and $G^{2}=G.G=\mathcal{V} |F_{t}| \mathcal{V^*} \,  \mathcal{V}
|F_{t}| \mathcal{V^*}=\mathcal{V} |F_{t}| \,
|F_{t}|\mathcal{V^*}=F_{t} F_{t}^*=|F_{t}^*|^{2}$, i.e.
$G=|F_{t}^*|=\mathcal{V} |F_{t}| \mathcal{V^*}$. Thus
$F_{t}^*=|F_{t}|\mathcal{V^*}=\mathcal{V^*} \mathcal{V} |F_{t}|
\mathcal{V^*}= \mathcal{V^*} |F_{t}^*|$. Following the proof of
Theorem \ref{thm_polar_decomp}, (ii) $ \Rightarrow$ (i), we get
$t^*=\mathcal{V^*}|t^*|$. The converse direction can be shown
taking into account that ${\mathcal V}^{**}=\mathcal V$ and
interchanging the roles of $t$ and $t^*$ in the first part of the
proof.
\end{proof}

\begin{corollary}
If $t \in R(E,F)$ has closed range, then $t$ has
polar decomposition. In this case the generalized inverse of $t$
is a bounded adjointable operator.
\end{corollary}

\begin{proof}If $t$ has closed range then Proposition 1.2 of
\cite{F-S} implies that $E=Ran(t^*) \oplus Ker(t)$ and $F=Ran(t)
\oplus Ker(t^*)$, so $t$ has polar decomposition by Theorem
\ref{thm_polar_decomp}.
The operators $s$ and $ \tilde{s}$ were defined in part ''(ii)
$\Rightarrow $ (iii)''. They are bounded because $Dom(s)=Ran(t)
\oplus Ker(t^*)=F$ and $Dom( \tilde{s} )=Ran(t^*) \oplus Ker(t)=E$,
i.e.~the generalized inverse of $t$ is a bounded adjointable
operator.
\end{proof}

Theorem \ref{thm_polar_decomp} and Corollary \ref{cor_transfer}
together with a recent result by Lun Chuan Zhang \cite{Zhang} and
by Qingxiang Xu and Lijuan Sheng \cite{Xu/Sheng} give
us the opportunity to derive a criterion for bounded $C^*$-linear
operators between Hilbert $C^*$-modules to admit a generalized
inverse in the sense of Banach algebra theory. These authors proved
independently that a bounded adjointable $C^*$-linear operator between
two Hilbert $C^*$-modules admits a bounded generalized inverse if
and only if the operator has closed range.

\begin{proposition}
Let $T \in R(E,F)$ be a bounded $\mathcal{A}$-linear operator
between two Hilbert  ${\mathcal{A}}$-modules $E,F$ over some
fixed $C^*$-algebra ${\mathcal{A}}$. Suppose, $T$ has polar
decomposition. Then $T$ admits a regular
operator $s$ as its generalized inverse. The converse implication
is also true. Moreover, the generalized inverse $s$ is
bounded if and only if the range of $T$ is closed.
\end{proposition}

\begin{corollary}
For $t \in R(E,F)$ the bounded transform $F_t$ has a bounded
generalized inverse if and only if $F_t$ has closed range, if and
only if $t$ has closed range.
\end{corollary}


Magajna and Schweizer have shown, respectively, that
$C^*$-algebras of compact operators can be characterized by the
property that every norm closed (coinciding with its biorthogonal
complement, respectively) submodule of every Hilbert $C^*$-module
over them is automatically an orthogonal summand, cf. \cite{MAG},
\cite{SCH}. Recently further generic properties of the category
of Hilbert $C^*$-modules over $C^*$-algebras which characterize
precisely the $C^*$-algebras of compact operators have been found
by the authors in \cite{FR1} and \cite{F-S}. All in all,
$C^*$-algebras of compact operators turn out to be of unique
interest in Hilbert $C^*$-module theory.

\begin{theorem} \label{thm-3-7}
Let $\mathcal{A}$ be a $C^*$-algebra. The
following conditions are equivalent, among others:
\begin{list}{(\roman{cou001})}{\usecounter{cou001}}

\item $\mathcal{A}$ is an arbitrary $C^*$-algebra of compact operators.

\item For every (maximal) norm closed left ideal $I$
of $ \mathcal{A}$ the corresponding open projection $p \in
\mathcal{A}^{**}$ is an element of the multiplier $C^*$-algebra
$M(\mathcal{A})$ of $ \mathcal{A}$.

\item For every Hilbert $\mathcal{A}$-module $E$ every
Hilbert $\mathcal{A}$-submodule $F \subseteq E$ is automatically
orthogonally complemented, i.e. $F$ is an orthogonal summand.

\item For every Hilbert $\mathcal{A}$-module $E$ every Hilbert
$\mathcal{A}$-submodule $F \subseteq E$ that coincides
with its biorthogonal complement $F^{\perp \perp}
\subseteq E$ is automatically orthogonally complemented in $E$.

\item For every Hilbert $ \mathcal{A}$-module $E$
every Hilbert $ \mathcal{A}$-submodule is automatically
topologi\-cally complemented there, i.e. it is a topological
direct summand.

\item For every pair of Hilbert $ \mathcal{A}$-modules
$E, F$, every densely defined closed operator $t:
Dom(t)\subseteq E \to F$ possesses a densely defined
adjoint operator $t^*: Dom(t^*)\subseteq  F \to E$.

\item For every pair of Hilbert $ \mathcal{A}$-modules
$E, F$, every densely defined closed operator $t:
Dom(t)\subseteq E \to F$ is regular.

\item The kernels of all densely defined closed
operators between arbitrary Hilbert $\mathcal{A}$-modules
are orthogonal summands.

\item The images of all densely defined closed operators
with norm closed range between ar\-bitrary Hilbert
$\mathcal{A}$-modules are orthogonal summands.
\end{list}
\end{theorem}

\begin{corollary}
Let $\mathcal{A}$ be a $C^*$-algebra. The
following conditions are equivalent:
\begin{list}{(\roman{cou001})}{\usecounter{cou001}}
\item $\mathcal{A}$ is an arbitrary $C^*$-algebra of
compact operators.
\end{list}
\begin{list}{(\roman{cou001})}{\usecounter{cou001}}
   \setcounter{cou001}{9}
\item For every pair of Hilbert $ \mathcal{A}$-modules
$E, F$, every densely defined closed operator $t: Dom(t)\subseteq
E \to F$ has polar decomposition, i.e. there exists a unique
partial isometry $ \mathcal{V}$ with initial set $
\overline{Ran(|t|)}$ and the final set $ \overline{Ran(t)}$ such
that $t=\mathcal{V}|t|$.

\item For every pair of Hilbert $ \mathcal{A}$-modules
$E, F$, every densely defined closed operator $t: Dom(t)\subseteq
E \to F$ and its adjoint have generalized inverses.
\end{list}
\end{corollary}

\begin{proof}The statements are deduced from
Theorem \ref{thm_polar_decomp} and from the conditions (vii), (viii)
of Theorem \ref{thm-3-7}.
\end{proof}

After searching in A.M.S.' MathSciNet data base we believe
that part of the results of Theorem 3.1, Proposition 3.5, and Corollaries
3.4, 3.6 and 3.8 are essentially new even in the case of Hilbert
spaces.

\medskip
{\bf Acknowledgement}: This research was done while the second
author stayed at the Ma\-the\-matisches Institut of the Westf\"alische
Wilhelms-Universit\"at in M\"unster, Germany, in 2007. He would like to
express his thanks to professor Michael Joachim and his
colleagues in the topology and operator algebras groups for
warm hospitality and great scientific atmosphere.
The authors would like to thank the referee for his/her very
useful comments which helped to improve parts of the proof
of Theorem 3.1 and to shorten the presentation of some facts.

\end{document}